\documentclass[12pt]{amsart}
\usepackage{amssymb}
\usepackage{amsmath, amscd}
\usepackage{amsthm}
\usepackage[table,xcdraw]{xcolor}
\usepackage{ulem}
\usepackage{tikz}
\usepackage{circuitikz}
\usetikzlibrary{positioning}
\usepackage[colorlinks=true, linkcolor=blue,urlcolor=blue]{hyperref}

\newtheorem{theorem}{Theorem}[section]

\newtheorem{proposition}[theorem]{Proposition}
\newtheorem{lemma}[theorem]{Lemma}

\newtheorem{corollary}[theorem]{Corollary}
\theoremstyle{definition}

\newtheorem{example}[theorem]{Example}
\newtheorem{definition}[theorem]{Definition}

\newtheorem{conjecture}[theorem]{Conjecture}
\newtheorem{remark}[theorem]{Remark}

\newtheorem{problem}[theorem]{Problem}

\newtheorem*{conjecture*}{Conjecture}

\begin{document}

\author[P. Danchev]{Peter Danchev}
\address{Institute of Mathematics and Informatics, Bulgarian Academy of Sciences, 1113 Sofia, Bulgaria}
\email{danchev@math.bas.bg; pvdanchev@yahoo.com}
\author[O. Hasanzadeh]{Omid Hasanzadeh}
\address{Department of Mathematics, Tarbiat Modares University, 14115-111 Tehran Jalal AleAhmad Nasr, Iran}
\email{hasanzadeomiid@gmail.com}
\author[A. Javan]{Arash Javan}
\address{Department of Mathematics, Tarbiat Modares University, 14115-111 Tehran Jalal AleAhmad Nasr, Iran}
\email{a.darajavan@gmail.com}
\author[A. Moussavi]{Ahmad Moussavi}
\address{Department of Mathematics, Tarbiat Modares University, 14115-111 Tehran Jalal AleAhmad Nasr, Iran}
\email{moussavi.a@modares.ac.ir; moussavi.a@gmail.com}

\title[Rings whose units are weakly nil-clean]{\small Rings whose invertible elements are \\ weakly nil-clean}




\begin{abstract}
We study those rings in which all invertible elements are weakly nil-clean calling them {\it UWNC rings}. This somewhat extends results due to Karimi-Mansoub et al. in Contemp. Math. (2018), where rings in which all invertible elements are nil-clean were considered abbreviating them as {\it UNC rings}. Specifically, our main achievements are that the triangular matrix ring ${\rm T}_n(R)$ over a ring $R$ is UWNC precisely when $R$ is UNC. Besides, the notions UWNC and UNC do coincide when $2 \in J(R)$. We also describe UWNC $2$-primal rings $R$ by proving that $R$ is a ring with $J(R) = {\rm Nil}(R)$ such that $U(R)=\pm 1+{\rm Nil}(R)$. In particular, the polynomial ring $R[x]$ over some arbitrary variable $x$ is UWNC exactly when $R$ is UWNC. Some other relevant assertions are proved in the present direction as well.
\end{abstract}

\maketitle

\section{Introduction and Major Concepts}

Everywhere in the current paper, let $R$ be an associative but {\it not} necessarily commutative ring having identity element, usually stated as $1$. Standardly, for such a ring $R$, the letters $U(R)$, $\rm{Nil}(R)$ and ${\rm Id}(R)$ are designed for the set of invertible elements (also termed as the unit group of $R$), the set of nilpotent elements and the set of idempotent elements in $R$, respectively. Likewise, $J(R)$ denotes the Jacobson radical of $R$, and ${\rm Z}(R)$ denotes the center of $R$. The ring of $n\times n$ matrices over $R$ and the ring of $n\times n$ upper triangular matrices over $R$ are denoted by ${\rm M}_{n}(R)$ and ${\rm T}_{n}(R)$, respectively. Standardly, a ring is said to be {\it abelian} if each of its idempotents is central, that is, ${\rm Id}(R)\subseteq {\rm Z}(R)$.

For all other unexplained explicitly notions and notations, we refer to the classical source \cite{L} or to the cited in the bibliography research sources. However, for completeness of the exposition and for the reader's convenience, we recall the following basic notions.

\begin{definition}[\cite{1}]\label{defi1.1}
Let $R$ be a ring. An element $r \in R$ is said to be {\it nil-clean} if there is an idempotent $e \in R$ and a nilpotent $b \in R$ such that $r=e+b$. Such an element $r$ is further called {\it strongly nil-clean} if the existing idempotent and nilpotent can be chosen such that $be=eb$. A ring is called {\it nil-clean} (respectively, {\it strongly nil-clean}) if each of its elements is nil-clean (respectively, strongly nil-clean).
\end{definition}

\begin{definition}[\cite{DM},\cite{10},\cite{2}]\label{defi1.2}
A ring $R$ is said to be {\it weakly nil-clean} provided that, for any $a \in R$, there exists an idempotent $e \in R$ such that $a-e$ or $a+e$ is nilpotent. A ring $R$ is said to be {\it strongly weakly nil-clean} provided that, for any $a \in R$, $a$ or $-a$ is strongly nil-clean.
\end{definition}

\begin{definition}[\cite{3},\cite{DL}]\label{defi1.3}
A ring is called {\it UU} if all of its units are unipotent,that is, $U(R) \subseteq 1+{\rm Nil}(R)$ (and so $1+{\rm Nil}(R)=U(R)$).
\end{definition}

\begin{definition}[\cite{4}]\label{defi1.4}
A ring $R$ is called {\it weakly UU} and abbreviated as $WUU$ if $U(R)={\rm Nil}(R) \pm 1$. This is equivalent to the condition that every unit can be presented as either $n+1$ or $n-1$, where $n \in {\rm Nil}(R)$.
\end{definition}

\begin{definition}[\cite{5}]\label{defi1.5}
A ring $R$ is called {\it UNC} if every of its units is nil-clean.
\end{definition}

Our key working instrument is the following one.

\begin{definition}\label{defi1.6}
A ring $R$ is called {\it UWNC} if every of its units is weakly nil-clean.
\end{definition}

In \cite{5}, the authors investigated UNC rings, i.e., those rings whose units are nil-clean. Our plan is to expand this substantially by developing an useful for this purpose machinery. In fact, we use some non-standard techniques from ring and matrix theories as well as also compare and distinguish the established results with these from \cite{4}.

It is worth noticing that some closely related background material can also be found in the sources \cite{D} and \cite{M}.

\medskip

The next constructions that show some proper ring classes inclusions are worthwhile.

\begin{example}\label{exam1.7}
\begin{enumerate}
\item
Any nil-clean ring is weakly nil-clean, but the converse is {\rm not} true in general. For instance, $\mathbb{Z}_{3}$ is weakly nil-clean but is {\it not} nil-clean.
\item
Any UU ring is WUU, but the converse is {\it not} true in general. For instance, $\mathbb{Z}$ is a WUU ring but is {\it not} UU.
\item
Any UU ring and nil-clean ring are UNC, but the converse is {\it not} true. For instance, the direct sum $\mathbb{Z}_{2}[t]\oplus {\rm M}_{2}(\mathbb{Z}_{2})$ is a UNC ring which is neither UU nor nil-clean.
\item
Any WUU ring and weakly nil-clean ring are UWNC, but the converse is {\it not} true in general. For instance, the direct sum $\mathbb{Z}\oplus M_{2}(\mathbb{Z}_{2})$ is a UWNC ring which is neither WUU nor weakly nil-clean.
\item
Any UNC ring is UWNC, but the converse is {\it not} true in general. For instance, all of the rings $\mathbb{Z}_{3}$, $\mathbb{Z}_{6}$, $\mathbb{Z}_{3}\oplus {\rm M}_{2}(\mathbb{Z}_{2})$ are UWNC but are {\it not} UNC.

\vskip1.0pc

\end{enumerate}
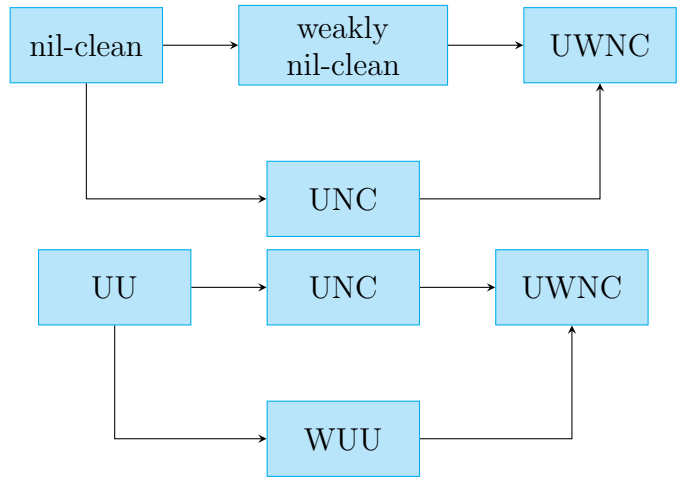
\begin{center}
\begin{tikzpicture}
\node[draw=cyan,fill=cyan!25,minimum width=2cm,minimum height=1cm,text width=1.75cm,align=center]  (a) {nil-clean};
\node[draw=cyan,fill=cyan!25,minimum width=2cm,minimum height=1cm,text width=2.5cm,align=center,right=of a](b){weakly nil-clean};
\node[draw=cyan,fill=cyan!25,minimum width=2cm,minimum height=1cm,text width=1.75cm,align=center,below =of b](c){UNC};
\node[draw=cyan,fill=cyan!25,minimum width=2cm,minimum height=1cm,text width=1.75cm,align=center,right=of b](d){UWNC};
\draw[-stealth] (a) -- (b);
\draw[-stealth] (a) |- (c);
\draw[-stealth] (b) -- (d);
\draw[-stealth] (c) -| (d);
\end{tikzpicture}
\end{center}
\begin{center}
\begin{tikzpicture}
\node[draw=cyan,fill=cyan!25,minimum width=2cm,minimum height=1cm,text width=1.75cm,align=center]  (a) {UU};
\node[draw=cyan,fill=cyan!25,minimum width=2cm,minimum height=1cm,text width=1.75cm,align=center,right=of a](b){UNC};
\node[draw=cyan,fill=cyan!25,minimum width=2cm,minimum height=1cm,text width=1.75cm,align=center,below =of b](c){WUU};
\node[draw=cyan,fill=cyan!25,minimum width=2cm,minimum height=1cm,text width=1.75cm,align=center,right=of b](d){UWNC};
\draw[-stealth] (a) -- (b);
\draw[-stealth] (a) |- (c);
\draw[-stealth] (b) -- (d);
\draw[-stealth] (c) -| (d);
\end{tikzpicture}
\end{center}
\end{example}

\vskip1.0pc

Our further programm is the following: In the next second section, we obtain some main properties of the newly defined class of UWNC rings -- the main results here are Proposition~\ref{prop2.10} as well as Theorems~\ref{theo0.2.2new}, \ref{theo2.14} and \ref{theo0.2.3new}, respectively. In the subsequent third section, we explore UWNC group rings -- the main results here are Propositions~\ref{prop0.3.1} and \ref{prop0.3.2}, respectively.

\section{Basic Properties of UWNC Rings}

We start here with the following obvious technicality.

\begin{proposition}\label{prop2.1}
A unit $u$ of a ring $R$ is strongly weakly nil-clean if, and only if, $u\in \pm 1+{\rm Nil}(R)$. In particular, $R$ is a WUU ring if, and only if, every unit of $R$ is strongly weakly nil-clean.
\end{proposition}


We continue our work with the next two technical claims as follows.

\begin{proposition}\label{prop2.3}
Let $R$ be a UWNC ring and $S$ a UNC ring. Then, $R\times S$ is a UWNC ring.
\end{proposition}

\begin{proof}
Choose $(u,v) \in U(R \times S)=U(R) \times U(S)$. Thus, there exist an idempotent $e \in R$ and a nilpotent $n \in R$ such that $u=e+n$ or $u=-e+n$. We now differ two cases:

\medskip

\noindent {\bf Case I.} Write $u=e+n$. Then, we have an idempotent $f \in S$ and a nilpotent $n^{\prime} \in S$ such that $v=f+n^{\prime}$. Thus, $(u, v)=(e, f)+(n, n^{\prime})$, where $(e, f) \in {\rm Id}(R \times S)$ and $(n, n^{\prime}) \in {\rm Nil}(R \times S)$.

\medskip

\noindent {\bf Case II.} Write $u=-e+n$. Then, we have an idempotent $f\in S$ and a nilpotent $n^{\prime} \in S$ such that $-v=f+ n^{\prime}$. Thus, $(u, v)=-(e, f)+(n,-n^{\prime})$, where $(e, f) \in {\rm Id}(R\times S)$ and $(n,-n^{\prime}) \in {\rm Nil}(R \times S)$.

\medskip

Therefore, $(u, v)$ is either the sum or difference of a nilpotent and an idempotent in $R \times S$, whence we get the desired result.
\end{proof}

\begin{proposition}\label{prop2.4}
Let $\{R_i\}$ be a family of rings. Then, the direct product $R=\prod R_i$ of rings $R_i$ is UWNC if, and only if, each $R_i$ is UWNC and at most one of them is {\it not} UNC.
\end{proposition}

\begin{proof}
($\Rightarrow$). Obviously, each $R_i$ is UWNC. Suppose now $R_{i_1}$ and $R_{i_2}$ $(i_1 \neq i_2)$ are {\it not} UNC. Then, there exist some $u_{i_j} \in U (R_{i_j}) \quad (j=1,2)$ such that $u_{i_1} \in U (R_{i_1})$ and $-u_{i_2} \in U (R_{i_2})$ are both {\it not} nil-clean. Choosing $u=(u_i)$, where $u_i=0$ whenever $i \neq i_{j}\quad (j=1,2)$, we infer that $u$ and $-u$ are {it not} the sum of an idempotent and a nilpotent, as required to get a contradiction. Consequently, each $R_i$ is a UWNC ring as at most one of them is {\it not} UNC.\\
($\Leftarrow$). Assume that $R_{i_0}$ is a UWNC ring and all the others $R_i$ are UNC. So, a simple check gives that $\prod_{i \neq i_{0}} R_i$ is UNC. According to Proposition \ref{prop2.3}, we conclude that $R$ is a UWNC ring.
\end{proof}

However, the property of being UWNC is {\it not} closed under taking (internal, external) direct sum as the next construction illustrate.

\begin{example}\label{exam0.2.1new}
A ring $\mathbb{Z}_{3}$ is a UWNC ring, but $\mathbb{Z}_{3}\times \mathbb{Z}_{3}$ is {\it not} UWNC.
\end{example}

Three further helpful affirmations are the following.

\begin{corollary}\label{cor2.5}
Let $L=\prod_{i \in I} R_i$ be the direct product of rings $R_i \cong R$ and $|I| \geq 2$. Then, $L$ is a UWNC ring if, and only if, $L$ is a UNC ring if, and only if, $R$ is a UNC ring.
\end{corollary}

\begin{corollary}\label{cor2.6}
For any $n \geq 2$, the ring $R^n$ is UWNC if, and only if, $R^n$ is UNC if, and only if, $R$ is UNC.
\end{corollary}

\begin{proposition}\label{prop0.2.5}
Let $R$ be a UWNC ring. If $T$ is a factor-ring of $R$ such that all units of $T$ lift to units of $R$, then $T$ is a UWNC ring.
\end{proposition}

\begin{proof}
Suppose that $f:R\rightarrow T$ is a surjective ring homomorphism. Let $v\in U(T)$. Then, there exists $u\in U(R)$ such that $v=f(u)$ and $u= \pm e+n$, where $e\in {\rm Id}(R)$ and $n\in {\rm Nil}(R)$. Therefore, we have $v=\pm f(e)+f(n)$, where $f(e)\in {\rm Id}(T)$ and $f(n) \in {Nil}(T)$, as needed.
\end{proof}

We now offer the validity of the following statement.

\begin{theorem}\label{theo2.7}
Let $R$ be a ring and $I$ a nil-ideal of $R$.
\begin{enumerate}
\item
$R$ is a UWNC ring if, and only if, $J(R)$ is nil and $\dfrac{R}{J(R)}$ is a UWNC ring.
\item
$R$ is a UWNC ring if, and only if, $\dfrac{R}{I}$ is a UWNC ring.
\end{enumerate}
\end{theorem}

\begin{proof}
\begin{enumerate}
\item
Let $R$ be a UWNC ring and suppose $x \in J(R)$ and $x \notin {\rm Nil}(R)$. Since $1+x \in U(R)$, it must be that $1+x=-e+n$, where $n \in {\rm Nil}(R)$ and $e \in {\rm Id}(R)$, because if $1+x=e+n$, we will have $x \in {\rm Nil}(R)$ that is a contradiction.

So, $2+x \in {\rm Nil}(R)$. Similarly, since $1+x^2 \in U(R)$, we deduce that $2+x^2 \in {\rm Nil}(R)$. Hence, $$(2+x^2)-(2+x)=x^2-x=-x(1-x) \in {\rm Nil}(R).$$ But $1-x \in U(R)$ whence $x \in {\rm Nil}(R)$, a contradiction. Thus, $J(R)$ is nil.

Now, letting $\bar{u} \in U\left(\bar{R}=\dfrac{R}{J(R)}\right)$, we obtain $u \in U(R)$, because units lift module $J(R)$. Therefore, $u= \pm e+n$, where $e \in {\rm Id}(R)$ and $n \in {\rm Nil}(R)$. So, we have $\bar{u}= \pm \bar{e}+\bar{n}$, where $\bar{e} \in {\rm Id}(\bar{R})$ and $\bar{n} \in {\rm Nil}(\bar{R})$. Thus, $\dfrac{R}{J(R)}$ is a UWNC ring, as promised.

Conversely, let $u\in U(R)$. Then, $\bar{u}\in U(\bar{R})$ and write $\bar{u}=\pm \bar{e}+\bar{n}$, where $\bar{e}\in {\rm Id}(\bar{R})$ and $\bar{n}\in {\rm Nil}(\bar{R})$. As $J(R)$ is nil, idempotents of $\dfrac{R}{J(R)}$ can be lifted to idempotents of $R$. So, we can assume that $e^{2}=e\in R$. Moreover, one inspects that $n\in R$ is nilpotent. Thus, for some $j\in J(R)$, $$u=\pm e+n+j=\pm e+(n+j)$$ is weakly nil-clean, because $n+j\in {\rm Nil}(R)$, as expected.
\item
The proof is similar to (i), so we omit the details.
\end{enumerate}
\end{proof}

Given a ring $R$ and a bi-module $_{R}M_{R}$, the trivial extension of $R$ by $M$ is the ring $T(R,M)=R\oplus M$ with the usual addition and the following multiplication: $(r_{1},m_{1})(r_{2},m_{2})= (r_{1}r_{2},r_{1}m_{2}+m_{1}r_{2})$. This is isomorphic to the ring of all matrices $\begin{pmatrix}
r & m\\
0 & r
\end{pmatrix}$, where $r\in R$ and $m\in M$ and the usual matrix operation are used.

\medskip

As an immediate consequence, we yield:

\begin{corollary}\label{cor2.8}
Let $R$ be a ring and $M$ a bi-module over $R$. Then, the following hold:
\begin{enumerate}
\item
The trivial extension ${\rm T}(R, M)$ is a UWNC ring if, and only if, $R$ is a UWNC ring.
\item
For $n \geq 2$, the quotient-ring $\dfrac{R[x]}{\langle x^n\rangle}$ is a UWNC ring if, and only if, $R$ is a UWNC ring.
\item
For $n \geq 2$, the quotient-ring $\dfrac{R[[x]]}{\langle x^n\rangle}$ is a UWNC ring if, and only if, $R$ is a UWNC ring.
\end{enumerate}
\end{corollary}

\begin{proof}
\begin{enumerate}
\item
Set $A={\rm T}(R, M)$ and consider $I:={\rm T}(0, M)$. It is not too hard to verify that $I$ is a nil-ideal of $A$ such that $\dfrac{A}{I} \cong R$. So, the result follows directly from Theorem \ref{theo2.7}.
\item
Put $A=\dfrac{R[x]}{\langle x^n\rangle}$. Considering $I:=\dfrac{\langle x\rangle}{\langle x^n\rangle}$, we obtain that $I$ is a nil-ideal of $A$ such that $\dfrac{A}{I} \cong R$. So, the result follows automatically Theorem \ref{theo2.7}.
\item
Knowing that the isomorphism $\dfrac{R[x]}{\langle x^n\rangle} \cong \dfrac{R[[x]]}{\langle x^n\rangle}$ is true, point (iii) follows automatically from (ii).
\end{enumerate}
\end{proof}

Consider $R$ to be a ring and $M$ to be a bi-module over $R$. Let $${\rm DT}(R,M) := \{ (a, m, b, n) | a, b \in R, m, n \in M \}$$ with addition defined componentwise and multiplication defined by $$(a_1, m_1, b_1, n_1)(a_2, m_2, b_2, n_2) = (a_1a_2, a_1m_2 + m_1a_2, a_1b_2 + b_1a_2, a_1n_2 + m_1b_2 + b_1m_2 +n_1a_2).$$ Then, ${\rm DT}(R,M)$ is a ring which is isomorphic to ${\rm T}({\rm T}(R, M), {\rm T}(R, M))$. Also, we have $${\rm DT}(R, M) =
\left\{\begin{pmatrix}
a &m &b &n\\
0 &a &0 &b\\
0 &0 &a &m\\
0 &0 &0 &a
\end{pmatrix} |  a,b \in R, m,n \in M\right\}.$$ We have the following isomorphism as rings: $\dfrac{R[x, y]}{\langle x^2, y^2\rangle} \rightarrow {\rm DT}(R, R)$ defined by $$a + bx + cy + dxy \mapsto
\begin{pmatrix}
a &b &c &d\\
0 &a &0 &c\\
0 &0 &a &b\\
0 &0 &0 &a
\end{pmatrix}.$$

We, thereby, detect the following.

\begin{corollary}
Let $R$ be a ring and $M$ a bi-module over $R$. Then the following statements are equivalent:
\begin{enumerate}
\item
$R$ is a UWNC ring.
\item
${\rm DT}(R, M)$ is a UWNC ring.
\item
${\rm DT}(R, R)$ is a UWNC ring.
\item
$\dfrac{R[x, y]}{\langle x^2, y^2\rangle}$ is a UWNC ring.
\end{enumerate}
\end{corollary}

Another two consequences of interest are the following ones:

\begin{corollary}\label{cor2.9}
Let $R$, $S$ be rings and let $M$ be an $(R, S)$-bi-module. If $T=\begin{pmatrix}
R & M\\
0 & S
\end{pmatrix}$ is a UWNC ring, then both $R, S$ are UWNC rings. The converse holds provided one of the rings $R$ or $S$ is UNC and the other is UWNC.
\end{corollary}

\begin{proof}
Given $I:=\left(\begin{array}{ll}0 & M \\ 0 & 0\end{array}\right)$. A routine inspection shows that this is a nil-ideal in $T$ with $\dfrac{T}{I} \cong R \times S$. Therefore, the result follows applying Proposition \ref{prop2.4} and Theorem \ref{theo2.7}.
\end{proof}

Let $\alpha$ be an endomorphism of $R$ and $n$ a positive integer. It was defined by Nasr-Isfahani in \cite{17} the {\it skew triangular matrix ring} like this:

$${\rm T}_{n}(R,\alpha )=\left\{ \left. \begin{pmatrix}
a_{0} & a_{1} & a_{2} & \cdots & a_{n-1} \\
0 & a_{0} & a_{1} & \cdots & a_{n-2} \\
0 & 0 & a_{0} & \cdots & a_{n-3} \\
\ddots & \ddots & \ddots & \vdots & \ddots \\
0 & 0 & 0 & \cdots & a_{0}
\end{pmatrix} \right| a_{i}\in R \right\}$$
with addition point-wise and multiplication given by:
\begin{align*}
&\begin{pmatrix}
a_{0} & a_{1} & a_{2} & \cdots & a_{n-1} \\
0 & a_{0} & a_{1} & \cdots & a_{n-2} \\
0 & 0 & a_{0} & \cdots & a_{n-3} \\
\ddots & \ddots & \ddots & \vdots & \ddots \\
0 & 0 & 0 & \cdots & a_{0}
\end{pmatrix}\begin{pmatrix}
b_{0} & b_{1} & b_{2} & \cdots & b_{n-1} \\
0 & b_{0} & b_{1} & \cdots & b_{n-2} \\
0 & 0 & b_{0} & \cdots & b_{n-3} \\
\ddots & \ddots & \ddots & \vdots & \ddots \\
0 & 0 & 0 & \cdots & b_{0}
\end{pmatrix}  =\\
& \begin{pmatrix}
c_{0} & c_{1} & c_{2} & \cdots & c_{n-1} \\
0 & c_{0} & c_{1} & \cdots & c_{n-2} \\
0 & 0 & c_{0} & \cdots & c_{n-3} \\
\ddots & \ddots & \ddots & \vdots & \ddots \\
0 & 0 & 0 & \cdots & c_{0}
\end{pmatrix},
\end{align*}
where $$c_{i}=a_{0}\alpha^{0}(b_{i})+a_{1}\alpha^{1}(b_{i-1})+\cdots +a_{i}\alpha^{i}(b_{i}),~~ 1\leq i\leq n-1
.$$ We denote the elements of ${\rm T}_{n}(R, \alpha)$ by $(a_{0},a_{1},\ldots , a_{n-1})$. If $\alpha $ is the identity endomorphism, then ${\rm T}_{n}(R,\alpha )$ is a subring of upper triangular matrix ring ${\rm T}_{n}(R)$.

\medskip

All of the mentioned above guarantee the truthfulness of the following statement.

\begin{corollary}\label{cor0.2.5new}
Let $R$ be a ring. Then, the following are equivalent:
\begin{enumerate}
\item
$R$ is a UWNC ring.
\item
${\rm T}_{n}(R,\alpha )$ is a UWNC ring.
\end{enumerate}
\end{corollary}

\begin{proof}
Choose
$$I:=\left\{
\left.
\begin{pmatrix}
0 & a_{12} & \ldots & a_{1n} \\
0 & 0 & \ldots & a_{2n} \\
\vdots & \vdots & \ddots & \vdots \\
0 & 0 & \ldots & 0
\end{pmatrix} \right| a_{ij}\in R \quad (i\leq j )
\right\}.$$
Then, one easily verifies that $I^{n}=0$ and $\dfrac{{\rm T}_{n}(R,\alpha )}{I} \cong R$. Consequently, Theorem \ref{theo2.7} applies to get the wanted result.
\end{proof}

Let $\alpha$ be an endomorphism of $R$. We denote by $R[x,\alpha ]$ the {\it skew polynomial ring} whose elements are the polynomials over $R$, the addition is defined as usual, and the multiplication is defined by the equality $xr=\alpha (r)x$ for any $r\in R$. Thus, there is a ring isomorphism $$\varphi : \dfrac{R[x,\alpha]}{\langle x^{n}\rangle }\rightarrow {\rm T}_{n}(R,\alpha),$$ given by $$\varphi (a_{0}+a_{1}x+\ldots +a_{n-1}x^{n-1}+\langle x^{n} \rangle )=(a_{0},a_{1},\ldots ,a_{n-1})$$ with $a_{i}\in R$, $0\leq i\leq n-1$. So, one finds that ${\rm T}_{n}(R,\alpha )\cong \dfrac{R[x,\alpha ]}{\langle  x^{n}\rangle}$, where $\langle x^{n}\rangle$ is the ideal generated by $x^{n}$.

We, thus, extract the following claim.

\begin{corollary}\label{cor2.12}
Let $R$ be a ring with an endomorphism $\alpha$ such that $\alpha (1)=1$. Then, the following are equivalent:
\begin{enumerate}
\item
$R$ is a UWNC ring.
\item
$\dfrac{R[x,\alpha ]}{\langle x^{n}\rangle }$ is a UWNC ring.
\item
$\dfrac{R[[x,\alpha ]]}{\langle x^{n}\rangle }$ is a UWNC ring.
\end{enumerate}
\end{corollary}

The next assertion is pivotal.

\begin{proposition}\label{prop2.10}
Let $R$ be a ring. Then, the following are equivalent:
\begin{enumerate}
\item
$R$ is a UNC ring.
\item
${\rm T}_{n}(R)$ is a UNC ring for all $n \in \mathbb{N}$.
\item
${\rm T}_n(R)$ is a UNC ring for some $n \in \mathbb{N}$.
\item
${\rm T}_n(R)$ is a UWNC ring for some $n \geq 2$.
\end{enumerate}
\end{proposition}

\begin{proof}
(i) $\Rightarrow$ (ii). This follows employing \cite[Corollary 2.6]{5}.\\
(ii) $\Rightarrow$ (iii) $\Rightarrow$ (iv). These two implications are trivial, so we remove the details.\\
(iv) $\Rightarrow$ (i). {\bf Method 1:} Let $u\in U(R)$ and choose
$$A=\begin{pmatrix}
u & & & \ast \\
  & -u_{1} & & \\
  & & \ddots & \\
  0 & & & 1
\end{pmatrix}\in U({\rm T}_n(R)).$$ By hypothesis, we can find an idempotent
$\begin{pmatrix}
e_{1} & & & \ast \\
  & e_{2} & & \\
  & & \ddots & \\
  0 & & & e_{n}
\end{pmatrix}$ and a nilpotent \\
$\begin{pmatrix}
w_{1} & & & \ast \\
  & w_{2} & & \\
  & & \ddots & \\
  0 & & & w_{n}
\end{pmatrix}$ such that
$$A=\pm \begin{pmatrix}
e_{1} & & & \ast \\
  & e_{2} & & \\
  & & \ddots & \\
  0 & & & e_{n}
\end{pmatrix}+\begin{pmatrix}
w_{1} & & & \ast \\
  & w_{2} & & \\
  & & \ddots & \\
  0 & & & w_{n}
\end{pmatrix}.$$
It now follows that $u=e_1+w_1$ or $u=e_2-w_2$. Clearly, $e_1$, $e_2$ are idempotents and $w_1$, $w_2$ are nilpotents in $R$, thus proving point (i).\\

\medskip

{\bf Method 2:} Setting $I:=\{ (a_{ij})\in T_{n}(R)| a_{ii}=0\}$, we obtain that it is a nil-ideal in ${\rm T}_{n}(R)$ with $\dfrac{{\rm T}_{n}(R)}{I}\cong R^{n}$. Therefore, Theorem \ref{theo2.7} and Corollary \ref{cor2.6} are applicable to get the pursued result.
\end{proof}

We know that the direct sum $\mathbb{Z}_{2}[x]\oplus {\rm M}_{2}(\mathbb{Z}_{2})$ is a UWNC ring that is neither WUU nor weakly nil-clean. In this vein, the following example concretely demonstrates an indecomposable UWNC ring that is neither WUU nor weakly nil-clean.

\begin{example}\label{exam0.2.1new}
Let $R={\rm M}_{n}(\mathbb{Z}_{2})$ with $n \geq 2$, $S=\mathbb{Z}_{2}[x]$ and $M=S^{n}$. Then, the formal triangular matrix ring
$T:=\begin{pmatrix}
R & M\\
0 & S
\end{pmatrix}$ is an indecomposable UNC ring invoking \cite[Example 2.7]{5}. So, $T$ is an indecomposable UWNC ring. But since $R$ is {\it not} a WUU ring and $S$ is {\it not} a weakly nil-clean ring, one plainly follows that $T$ is neither a WUU ring nor a weakly nil-clean ring, as claimed.
\end{example}

It was proved in \cite[Proposition 2.25]{4} that any unital subring of a WUU ring is again a WUU ring. But, curiously, a subring of a UWNC ring may {\it not} be a UWNC ring as the next example shows.

\begin{example}\label{exam0.2.2new}
Let $T={\rm M}_{2}(\mathbb{Z}_{2})$ and $u=\begin{pmatrix}
0 & 1\\
1 & 1
\end{pmatrix}\in T$. Then, one sees that $u^{3}=1$. Now, let $R$ be the unital subring of $T$ generated by $u$. Therefore, one calculates that $$R=\{ a1+bu+cu^{2}|a,b,c \in \mathbb{Z}\} =\{ 0,1,u,u^{2},1+u,1+u^{2},1+u+u^{2},u+u^{2}\}.$$ But, we have that $1+u=u^{2}$, $1+u^{2}=u$, $1+u+u^{2}=0$ and $u+u^{2}=1$, so we deduce $R=\{ 0,1,u,u^{2}\}$. It is now easy to see that ${\rm Nil}(R)=\{0\}$, so that $R$ is reduced. As $u^{2}\neq u$, $u$ is manifestly {\it not} weakly nil-clean utilizing \cite[Theorem 20]{10}, and so $T$ is a UWNC ring, as asserted.
\end{example}

\begin{proposition}\label{prop2.12}
The following two statements are valid:
\begin{enumerate}
\item
If $R$ is a weakly nil-clean ring, then ${\rm Z}(R)$ is strongly weakly nil-clean.
\item
If $R$ is a UWNC ring, then ${\rm Z}(R)$ is a WUU ring.
\end{enumerate}
\end{proposition}

\begin{proof}
\begin{enumerate}
\item
Let $a \in {\rm Z}(R)$. Then, $a \in R$ is weakly nil-clean and central, so $a$ is strongly weakly nil-clean in $R$. Thus, $a \pm a^2 \in {\rm Nil}(R)$ by \cite[Theorem 2.1]{6}. But, $a \pm a^2 \in {\rm Z}(R)$ so that $$a \pm a^2 \in {\rm Nil}(R) \cap {\rm Z}(R) \subseteq {\rm Nil}({\rm Z}(R)).$$ Hence, ${\rm Z}(R)$ is strongly weakly nil-clean, as required.
\item
The proof is analogous to (i).
\end{enumerate}
\end{proof}

\begin{proposition}\label{prop0.2.7}
For any ring $R$, the power series ring $R[[x]]$ is {\it not} UWNC.
\end{proposition}

\begin{proof}
Note the principal fact that the Jacobson radical of $R[[x]]$ is {\it not} nil (see, e.g., \cite{L}). Thus, in view of Theorem \ref{theo2.7}, $R[[x]]$ is really {\it not} a UWNC ring, as expected.
\end{proof}

\begin{lemma}\label{lem2.13}
Let $R$ be a ring. Then, the following two points are equivalent:
\begin{enumerate}
\item
$R$ is a UNC ring.
\item
$R$ is a UWNC ring and $2 \in J(R)$.
\end{enumerate}
\end{lemma}

\begin{proof}
(i) $\Longrightarrow$ (ii). Evidently, $R$ is a UWNC ring. Also, we have $2 \in J(R)$ in virtue of \cite[Lemma 2.4]{5}.\\
(ii) $\Rightarrow$ (i). Notice that $\dfrac{R}{J(R)}$ is of characteristic $2$, because $2 \in J(R)$, and so $a=-a$ for every $a \in \dfrac{R}{J(R)}$. That is why, $\dfrac{R}{J(R)}$ is a UNC ring, and thus we can apply \cite[Theorem 2.5]{5} since $J(R)$ is nil in view of Theorem \ref{theo2.7}.
\end{proof}

Recall that an element $r$ in a ring $R$ is said to be an {\it unipotent} if $r-1$ is a nilpotent. The following technical claim is elementary, but rather applicable in the sequel.

\begin{lemma}\label{lem0.2.2}
Let $R$ be a ring and let $r\in R$ be the sum of an idempotent and a nilpotent. If $r^{2}=1$, then $r$ is unipotent.
\end{lemma}

\begin{proof}
Write $r=e+n$ with $e\in {\rm Id}(R)$ and $n\in {\rm Nil}(R)$. Set $f:=1-e$ and $x:=n(n+1)\in {\rm Nil}(R)$. Taking into account the equality $fn=f(r-e)=fr$, we compute that $$fx=fn(n+1)=fr(r-e+1)=fr(r+f)=fr^{2}+frf=f+frf,$$ and, similarly, that $xf=f+frf$. Hence, $fx=xf$, so that $x$ is a nilpotent which commutes with $f$, $e$, $n$ and $r$, respectively. Accordingly, $$f=fr^{2}=fr\cdot r=fnr=fx(1+n)^{-1}r=f(1+n)^{-1}r\cdot x$$ is a nilpotent and hence $f=0$, as desired.
\end{proof}

Our next main result, which sounds quite surprisingly, is the following.

\begin{theorem}\label{theo0.2.2new}
Let $R$ be a ring and $2\in U(R)$. Then, the following two items are equivalent:
\begin{enumerate}
\item
$R$ is a UWNC ring.
\item
$R$ is a WUU ring.
\end{enumerate}
\end{theorem}

\begin{proof}
(ii) $\Rightarrow$ (i). This is pretty obvious, so we leave the argumentation.\\
(i) $\Rightarrow$ (ii). First, we show that $R$ is an abelian ring. To this goal, let $e^{2}=e\in R$, and let $a=1-2e$. Then, it is obviously true that $a^{2}=1$. Since $R$ is UWNC, either $a$ or $-a$ is nil-clean. By virtue of Lemma \ref{lem0.2.2}, one has that $a\in 1+{\rm Nil}(R)$ or $a\in -1+{\rm Nil}(R)$. If $a\in 1+{\rm Nil}(R)$, then $2e\in {\rm Nil}(R)$, and so $e\in {\rm Nil}(R)$. This implies that $e=0$. If, however, $a\in -1+{\rm Nil}(R)$, then $2(1-e)\in {\rm Nil}(R)$, whence $1-e\in {\rm Nil}(R)$. This forces that $e=1$. Therefore, $R$ has only trivial idempotents. Thus, $R$ is abelian, as asserted.

Now, let $u\in U(R)$, so $u=\pm e+n$, where $e\in {\rm Id}(R)$ and $n\in {\rm Nil}(R)$. If $u=e+n$, so $e=u-n\in U(R)$, then $e=1$. If, however, $u=-e+n$, so $e=u-n\in U(R)$, then $e=1$. Therefore, $u\in \pm 1+{\rm Nil}(R)$. Finally, one concludes that $R$ is WUU, as formulated.
\end{proof}

As an immediate consequence, we derive:

\begin{corollary}\label{cor0.2.5}
Let $R$ be an abelian ring. Then, the following are equivalent:
\begin{enumerate}
\item
$R$ is a UWNC ring.
\item
$R$ is a WUU ring.
\end{enumerate}
\end{corollary}

Appealing to \cite{4}, a commutative ring $R$ is a WUU ring if, and only if, so is $R[x]$. In what follows, we present a generalization of this result. Standardly, the prime radical ${\rm N}(R)$ of a ring $R$ is defined to be the intersection of the prime ideals of $R$. It is know that ${\rm N}(R)={\rm Nil}_{\ast}(R)$, the lower nil-radical of $R$. A ring $R$ is called a $2$-primal ring if ${\rm N}(R)$ coincides with ${\rm Nil}(R)$. For an endomorphism $\alpha$ of a ring $R$, $R$ is called {\it $\alpha$-compatible} if, for any $a,b\in R$, $ab=0\Longleftrightarrow a\alpha (b)=0$, and in this case $\alpha$ is clearly injective.

We now arrive at our third chief result.

\begin{theorem}\label{theo2.14}
Let $R$ be a 2-primal ring and $\alpha $ an endomorphism of $R$ such that $R$ is $\alpha$-compatible. The following issues are equivalent:
\begin{enumerate}
\item
$R[x, \alpha]$ is a UWNC ring.
\item
$R[x, \alpha]$ is a WUU ring.
\item
$R$ is a WUU ring.
\item
$R$ is a UWNC ring.
\item
$J(R)={\rm Nil}(R)$ and $U(R)=1 \pm J(R)$.
\end{enumerate}
\end{theorem}

\begin{proof}
(ii) $\Rightarrow$ (i) and (iii) $\Rightarrow$ (iv). Straightforward.\\
(i) $\Rightarrow$ (iv). As $\dfrac{R[x, \alpha]}{\langle x\rangle} \cong R$ and all units of $\dfrac{R[x, \alpha]}{\langle x\rangle}$ are lifted to units of $R[x, \alpha]$, the implication easily holds.\\
(ii) $\Longrightarrow$ (iii). We argue as in the proof of (i) $\Longrightarrow$ (iv).\\
(iv) $\Rightarrow$ (v). As $R$ is 2-primal, we have ${\rm Nil}(R) \subseteq J(R)$, so that $J(R)={\rm Nil}(R)$ bearing in mind Theorem \ref{theo2.7}. Let $a \in U(R)$, so by hypothesis we have $u= \pm e+n$, where $e \in {\rm Id}(R)$ and $n \in {\rm Nil}(R)=J(R)$. If $u=e+n$, so $e=u-n \in U(R)$, and thus $e=1$. If, however, $u=-e+n$, so $e=n-u \in U(R)$, and thus $e=1$. Therefore, we receive $u \in \pm 1+J(R)$, and hence $U(R)= \pm 1+J(R)$, as required.\\
(v) $\Rightarrow$ (ii). As $R$ is a $2$-primal ring, with the aid of (v) we have $J(R)={\rm Nil}(R)=\mathrm{Nil}_{\ast}(R)$. Thus, the quotient-ring $\dfrac{R}{J(R)}$ is a reduced ring. Moreover, it is easy to see that $\alpha ({\rm Nil}(R)) \subseteq {\rm Nil}(R)$, so $\alpha (J(R)) \subseteq J(R)$ and $\bar{\alpha}: \dfrac{R}{J(R)} \rightarrow \dfrac{R}{J(R)}$, defined by $\bar{\alpha}(\bar{a})=\overline{\alpha(a)}$, is an endomorphism of $\dfrac{R}{J(R)}$.

We next show that $\dfrac{R}{J(R)}$ is $\bar{\alpha}$-compatible. That is, we must show that, for any $a+J(R), b+J(R) \in \dfrac{R}{J(R)}$, the equivalence $$(a+J(R))(b+J(R))=J(R) \Leftrightarrow (a+J(R)) \bar{\alpha}(b+J(R))=J(R)$$ holds. Equivalently, we have to show that, for any $a, b \in R$, the equivalence $ab \in {\rm Nil}(R) \Leftrightarrow a \alpha(b) \in {\rm Nil}(R)$ is true. But this equivalence has been established in the proof of Claims 1 and 2 in \cite[Theorem 3.6]{7}. As $\dfrac{R}{J(R)}$ is a reduced factor-ring and also is $\bar{\alpha}$-compatible, with \cite[Corollary 2.12]{8} at hand we have $$U \left(\dfrac{R}{J(R)}[x, \bar{\alpha}]\right)=U \left(\dfrac{R}{J(R)}\right),$$ which is equal to $\{ \pm \bar{1}\}$ by assumption. So, $$\dfrac{R}{J(R)}[x,\bar{\alpha}] \cong \dfrac{R[x, \alpha]}{J(R)[x, \alpha]}$$ is a WUU ring. Also, \cite[Lemma 2.2]{8} tells us that
$$\mathrm{Nil}_*(R[x, \alpha])=\mathrm{Nil}_*(R)[x, \alpha].$$ Therefore, $$J(R)[x, \alpha]=\mathrm{Nil}_*(R[x, \alpha]),$$ which is manifestly nil. Hence, \cite[Proposition 2.5 and 2.6]{4} ensures that $R[x, \alpha]$ is a WUU ring, as asked for.
\end{proof}

As a direct consequence, we deduce:

\begin{corollary}\label{cor2.15}
Let $R$ be a 2-primal ring. Then, the following are equivalent:
\begin{enumerate}
\item
$R$ is a UWNC ring.
\item
$R[x]$ is a UWNC ring.
\item
$J(R)={\rm Nil}(R)$ and $U(R)= \pm 1+J(R)$.
\end{enumerate}
\end{corollary}

The following criterion is worthy of documentation.

\begin{proposition}\label{prop2.11}
Suppose $R$ is a commutative ring. Then, $R[x]$ is a UWNC ring if, and only if, $R$ is a UWNC ring.
\end{proposition}

\begin{proof}
\noindent{\bf Method 1:} For the necessity, let $R[x]$ be a UWNC ring. We know that $\dfrac{R[x]}{\langle x\rangle} \cong R$. It, therefore, suffices to show that $\dfrac{R[x]}{\langle x\rangle}$ is a UWNC ring. To this aim, choosing $u+\langle x\rangle \in U\left(\dfrac{R[x]}{\langle x\rangle}\right)$, we derive $u \in U(R[x])$, because all units of $\dfrac{R[x]}{\langle x\rangle}$ are lifted to units of $R[x]$. Thus, we write $u= \pm e+n$, where $e \in {\rm Id}(R[x])$ and $n \in {\rm Nil}(R[x])$. So, $$u+\langle x\rangle= \pm(e+\langle x\rangle)+(n+\langle x\rangle),$$ where $e+\langle x\rangle \in {\rm Id}\left(\dfrac{R[x]}{\langle x\rangle}\right)$ and $n+\langle x\rangle \in {\rm Nil}\left(\dfrac{R[x]}{\langle x\rangle}\right)$. Consequently, $u+\langle x\rangle$ is a weakly nil-clean element and $\dfrac{R[x]}{\langle  x\rangle }$ is a UWNC ring, as desired.

For the sufficiency, write $$f=a_0+a_1 x+\ldots + a_n x^n \in U (R[x]),$$ so $a_0 \in U (R)$ and $a_1,\ldots ,a_n \in {\rm Nil}(R)$. By hypothesis, we have $a_0= \pm e+n$, where $e \in {\rm Id}(R)$ and $n \in {\rm Nil}(R)$. This allows us to infer that $$f=( \pm e+n)+a_1 x+\cdots+a_n x^n= \pm e+(n+a_1 x+\cdots+a_n x^n),$$ where $e \in {\rm Id}(R[x])$ and $n+a_1 x+\cdots+a_n x^n \in {\rm Nil}(R[x])$. Finally, $f$ is a weakly nil-clean element, thus establishing the result.

\medskip

\noindent {\bf Method 2:} It is well know that every commutative ring is a $2$-primal ring, and hence the result can be deduced from Corollary \ref{cor2.15}.
\end{proof}

Incidentally, we are able to prove the following curious statement.

\begin{proposition}\label{prop0.2.6new}
Let $R$ be a ring, and $m,n\geq 1$. If the matrix rings ${\rm M}_{n}(R)$ and ${\rm M}_{m}(R)$ are both UWNC, then so is the triangular matrix ring ${\rm T}_{n+m}(R)$.
\end{proposition}

\begin{proof}
Let $V\in U ({\rm T}_{n+m}(R))$ be the $(n+m)\times (n+m)$ triangular matrix which we will write in the block decomposition form as follows $V=\begin{pmatrix}
V_{11} & A_{12} \\
0 & V_{22}
\end{pmatrix}$, where $V_{11}\in U({\rm M}_{n}(R))$, $V_{22}\in U({\rm M}_{m}(R))$ and $A_{12}$ is appropriately sized rectangular matrices. By hypothesis, there exist idempotent matrices $E_{1}$, $E_{2}$ and nilpotent matrices $N_{1}$, $N_{2}$ in ${\rm M}_{n}(R)$ and ${\rm M}_{m}(R)$ such that $V_{11}=\pm E_{1}+N_{1}$ and $V_{22}=\pm E_{2}+N_{2}$. Thus, we obtain the decomposition
\begin{align*}
\begin{pmatrix}
V_{11} & A_{12}\\
0 & V_{22}
\end{pmatrix}& =\begin{pmatrix}
\pm E_{1}+N_{1} & A_{12}\\
0 & \pm E_{2}+N_{2}
\end{pmatrix} \\
& =\pm \begin{pmatrix}
E_{1} & 0\\
0 & E_{2}
\end{pmatrix}+ \begin{pmatrix}
N_{1} & A_{12}\\
0 & N_{2}
\end{pmatrix}.
\end{align*}
Since $E_{1}$, $E_{2}$ are idempotents, then an easy verification guarantees that $\begin{pmatrix}
E_{1} & 0\\
0 & E_{2}
\end{pmatrix}$ is an idempotent. It is also readily to see that $\begin{pmatrix}
N_{1} & A_{12}\\
0 & N_{2}
\end{pmatrix}$ is a nilpotent. Thus, the above decomposition is the desired weakly nil-clean decomposition.
\end{proof}

Let $A$, $B$ be two rings and $M$, $N$ be $(A,B)$-bi-module and $(B,A)$-bi-module, respectively. Also, we consider the bilinear maps $\phi :M\otimes_{B}N\rightarrow A$ and $\psi:N\otimes_{A}M\rightarrow B$ that apply to the following properties.
$$Id_{M}\otimes_{B}\psi =\phi \otimes_{A}Id_{M},Id_{N}\otimes_{A}\phi =\psi \otimes_{B}Id_{N}.$$
For $m\in M$ and $n\in N$, define $mn:=\phi (m\otimes n)$ and $nm:=\psi (n\otimes m)$. Now the $4$-tuple $R=\begin{pmatrix}
A & M\\
N & B
\end{pmatrix}$ becomes to an associative ring with obvious matrix operations that is called a {\it Morita context ring}. Denote two-side ideals $Im \phi$ and $Im \psi$ to $MN$ and $NM$, respectively, that are called the {\it trace ideals} of the Morita context (compare with \cite{9} as well).

We now have at our disposal all the ingredients necessary to establish the following statement.

\begin{proposition}\label{prop2.16}
Let $R=\left(\begin{array}{ll}A & M \\ N & B\end{array}\right)$ be a Morita context ring such that $MN$ and $NM$ are nilpotent ideals of $A$ and $B$, respectively. If $R$ is a UWNC ring, then $A$ and $B$ are UWNC rings. The converse holds provided one of the $A$ or $B$ is UNC and the other is UWNC.
\end{proposition}

\begin{proof}
Canonically, $\dfrac{M}{M_{0}}$ is an $\left(\dfrac{A}{J(A)}, \dfrac{B}{J(B)}\right)$-bi-module and $\dfrac{N}{N_{0}}$ is a $\left(\dfrac{B}{J(B)}, \dfrac{A}{J(A)} \right)$-bi-module, and this induces a Morita context $\begin{pmatrix}
\dfrac{A}{J(A)} & \dfrac{M}{M_{0}} \\
\dfrac{N}{N_{0}} & \dfrac{B}{J(B)}
\end{pmatrix}$, where the context products are given by $$(x+M_{0})(y+N_{0})=xy+J(A), (y+N_{0})(x+M_{0})=yx+J(B)$$ for all $x\in M$ and $y\in N$. Thus, an appeal to \cite[Lemma 4.5]{9}, is a guarantor that $$\dfrac{R}{J(R)}\cong \begin{pmatrix}
\dfrac{A}{J(A)} & \dfrac{M}{M_{0}} \\
\dfrac{N}{N_{0}} & \dfrac{B}{J(B)}
\end{pmatrix}.$$ However, as $MN$ and $NM$ are nilpotent ideals of $A$ and $B$, respectively, we then have that $MN \subseteq J(A)$ and $NM\subseteq J(B)$. Therefore, in view of \cite{21}, we argue that $J(R)=\begin{pmatrix}
J(A) & M \\
N & J(B)
\end{pmatrix}$ and hence $\dfrac{R}{J(R)}\cong \dfrac{A}{J(A)}\times \dfrac{B}{J(B)}$.
Since $R$ is a UWNC ring, then the factor $\dfrac{R}{J(R)}$ is a UWNC ring and $J(R)$ is nil consulting with Theorem \ref{theo2.7}, so it follows that both $\dfrac{A}{J(A)}$ and $\dfrac{B}{J(B)}$ are UWNC. As $J(R)$ is nil, $J(A)$ and $J(B)$ are nil too. Thus, $A$ and $B$ are UWNC as well.

Oppositely, assuming that $A$ is a UNC ring and $B$ is a UWNC ring, we conclude that $\dfrac{R}{J(R)}$ is a UWNC ring by a combination of \cite[Theorem 2.5]{5}, Theorem \ref{theo2.7} and Proposition \ref{prop2.3}. It then suffices to show that $J(R)$ is nil. To this target, suppose
$r=\begin{pmatrix}
a & m\\
n & b
\end{pmatrix}\in J(R)$. Then, $a\in J(A)$, $b\in J(B)$. In virtue of Theorem \ref{theo2.7}, both ideals $J(A)$ and $J(B)$ are nil. Thus, we can find $n\in \mathbb{N}$ such that $a^{n}=0$ in $A$ and $b^{n}=0$ in $B$. So,
$$\begin{pmatrix}
a & m\\
n & b
\end{pmatrix}^{n+1}\subseteq \begin{pmatrix}
MN & M\\
N & NM
\end{pmatrix}.$$
Clearly,
$$ \begin{pmatrix}
MN & M\\
N & NM
\end{pmatrix}^{2}=\begin{pmatrix}
MN & (MN)M \\
(NM)N & NM
\end{pmatrix}.$$
Moreover, for any $j\in \mathbb{N}$, one easily checks that
$$\begin{pmatrix}
MN & M\\
N & NM
\end{pmatrix}^{2j}=\begin{pmatrix}
MN & (MN)M\\
(NM)N  & NM
\end{pmatrix}^{j}=\begin{pmatrix}
(MN)^{j} & (MN)^{j}M\\
(NM)^{j}N & (NM)^{j}
\end{pmatrix}.$$
By hypothesis, we may assume that $(MN)^{p}=0$ in $A$ and $(NM)^{p}=0$ in $B$. Therefore,
$$\begin{pmatrix}
MN & M\\
N & NM
\end{pmatrix}^{2p}=0.$$
Consequently,
$\begin{pmatrix}
a & m\\
n & b
\end{pmatrix}^{2p(n+1)}=0$, and so $J(R)$ is indeed nil, as desired.
\end{proof}



We now state the following.

\begin{example}
Consider the Morita context $R =
\begin{pmatrix}
\mathbb{Z}_4 & \mathbb{Z}_4\\
2\mathbb{Z}_4 & \mathbb{Z}_4
\end{pmatrix}$, where the context products are the same as the product in $\mathbb{Z}_4$. Then, we claim that $R$ is UWNC. Since $\mathbb{Z}_4$ is obviously UNC, so we are done exploiting Proposition \ref{prop2.16}. This substantiates our claim.
\end{example}

Given a ring $R$ and a central elements $s$ of $R$, the $4$-tuple $\begin{pmatrix}
R & R\\
R & R
\end{pmatrix}$ becomes a ring with addition component-wise and with multiplication defined by
$$\begin{pmatrix}
a_{1} & x_{1}\\
y_{1} & b_{1}
\end{pmatrix}\begin{pmatrix}
a_{2} & x_{2}\\
y_{2} & b_{2}
\end{pmatrix}=\begin{pmatrix}
a_{1}a_{2}+sx_{1}y_{2} & a_{1}x_{2}+x_{1}b_{2} \\
y_{1}a_{2}+b_{1}y_{2} & sy_{1}x_{2}+b_{1}b_{2}
\end{pmatrix}.$$
This ring is denoted by ${\rm K}_{s}(R)$. A Morita context
$\begin{pmatrix}
A & M\\
N & B
\end{pmatrix}$ with $A=B=M=N=R$ is called a {\it generalized matrix ring} over $R$. It was observed by Krylov in \cite{18} that a ring $S$ is a generalized matrix ring over $R$ if, and only if, $S={\rm K}_{s}(R)$ for some $s\in {\rm Z}(R)$. Here $MN=NM=sR$, so $MN\subseteq J(A)\Longleftrightarrow s\in J(R)$, $NM\subseteq J(B)\Longleftrightarrow s\in J(R)$, and $MN$, $NM$  are nilpotent $\Longleftrightarrow s$ is a nilpotent.

\medskip

As three corollaries, we extract:

\begin{corollary}\label{cor2.17}
Let $R$ be a ring and $s\in {\rm Z}(R)\cap {\rm Nil}(R)$. If ${\rm K}_{s}(R)$ is a UWNC ring, then $R$ is a UWNC ring. The converse holds, provided $R$ is a UNC ring.
\end{corollary}

An other construction of interest is the following one.

\begin{example}
As $\mathbb{Z}_4$ is a UNC ring, it follows from Corollary \ref{cor2.17} that the generalized matrix ring ${\rm K}_{(2)} (\mathbb{Z}_4) = \{
\begin{pmatrix}
a & b\\
c & d
\end{pmatrix} | a, b, c, d \in \mathbb{Z}_4 \}$ is a UWNC ring.
\end{example}

Following Tang and Zhou (cf. \cite{19}), for $n\geq 2$ and for $s\in {\rm Z}(R)$, the $n\times n$ formal matrix ring over $R$ defined by $s$, and denoted by ${\rm M}_{n}(R;s)$, is the set of all $n\times n$ matrices over $R$ with usual addition of matrices and with multiplication defined below:

\noindent For $(a_{ij})$ and $(b_{ij})$ in ${\rm M}_{n}(R;s)$,
$$(a_{ij})(b_{ij})=(c_{ij}), \quad \text{where} ~~ (c_{ij})=\sum s^{\delta_{ikj}}a_{ik}b_{kj}.$$
Here, $\delta_{ijk}=1+\delta_{ik}-\delta_{ij}-\delta_{jk}$, where $\delta_{jk}$, $\delta_{ij}$, $\delta_{ik}$ are the Kroncker delta symbols.

We, therefore, have the following.

\begin{corollary}\label{cor0.2.7}
Let $R$ be a ring and $s\in {\rm Z}(R)\cap {\rm Nil}(R)$. If ${\rm M}_{n}(R;s)$ is a UWNC ring, then $R$ is a UWNC ring. The converse holds, provided $R$ is a UNC ring.
\end{corollary}

\begin{proof}
If $n = 1$, then ${\rm M}_n(R;s) = R$. So, in this case, there is nothing to prove. Let $n=2$. By the definition of ${\rm M}_n(R;s)$, we have ${\rm M}_2 (R;s) \cong {\rm K}_{s^2} (R)$. Apparently, $s^2 \in {\rm Nil} (R) \cap {\rm Z} (R)$, so the claim holds for $n = 2$ with the help of Corollary \ref{cor2.17}.

To proceed by induction, assume now that $n>2$ and that the claim holds for ${\rm M}_{n-1} (R;s)$. Set $A := {\rm M}_{n-1} (R;s)$. Then, ${\rm M}_n (R;s) =
\begin{pmatrix}
A & M\\
N & R
\end{pmatrix}$
is a Morita context, where $$M =
\begin{pmatrix}
M_{1n}\\
\vdots\\
M_{n-1, n}
\end{pmatrix}
\quad \text{and} \quad  N = (M_{n1} \dots M_{n, n-1})$$ with $M_{in} = M_{ni} = R$ for all $i = 1, \dots, n-1,$ and
\begin{align*}
&\psi: N \otimes M \rightarrow N, \quad n \otimes m \mapsto snm\\
&\phi : M \otimes N \rightarrow M, \quad  m \otimes n \mapsto smn.
\end{align*}
Besides, for $x =
\begin{pmatrix}
x_{1n}\\
\vdots\\
x_{n-1, n}
\end{pmatrix}
\in M$ and $y = (y_{n1} \dots y_{n, n-1}) \in N$, we write $$xy =
\begin{pmatrix}
s^2x_{1n}y_{n1} & sx_{1n}y_{n2} & \dots & sx_{1n}y_{n, n-1}\\
sx_{2n}y_{n1} & s^2x_{2n}y_{n2} & \dots & sx_{2n}y_{n, n-1}\\
\vdots & \vdots &\ddots & \vdots\\
sx_{n-1, n}y_{n1} & sx_{n-1, n}y_{n2} & \dots & s^2x_{n-1, n}y_{n, n-1}
\end{pmatrix} \in sA$$ and $$yx = s^2y_{n1}x_{1n} + s^2y_{n2}x_{2n} + \dots + s^2y_{n, n-1}x_{n-1, n} \in s^2 R.$$ Since $s$ is nilpotent, we see that $MN$ and $NM$ are nilpotent too. Thus, we obtain that $$\frac{{\rm M}_n (R; s)}{J({\rm M}_n (R; s))} \cong \frac{A}{J (A)} \times \frac{R}{J (R)}.$$ Finally, the induction hypothesis and Proposition \ref{prop2.16} yield the claim after all.
\end{proof}

A Morita context $\begin{pmatrix}
A & M\\
N & B
\end{pmatrix}$ is called {\it trivial}, if the context products are trivial, i.e., $MN=0$ and $NM=0$. We now have
$$\begin{pmatrix}
A & M\\
N & B
\end{pmatrix}\cong {\rm T}(A\times B, M\oplus N),$$
where
$\begin{pmatrix}
A & M\\
N & B
\end{pmatrix}$ is a trivial Morita context consulting with \cite{20}.

What we can now offer is the following.

\begin{corollary}\label{cor0.2.8}
If the trivial Morita context
$\begin{pmatrix}
A & M\\
N & B
\end{pmatrix}$ is a UWNC ring, then $A$, $B$ are UWNC rings. The converse holds if one of the rings $A$ or $B$ is UWNC and the other is UNC.
\end{corollary}

\begin{proof}
It is apparent to see that the isomorphisms
$$\begin{pmatrix}
A & M\\
N & B
\end{pmatrix} \cong {\rm T}(A\times B,M\oplus N) \cong \begin{pmatrix}
A\times B & M\oplus N\\
0 & A \times B
\end{pmatrix}$$ are fulfilled. Then, the rest of the proof follows combining Corollary \ref{cor2.8} and Proposition \ref{prop2.4}.
\end{proof}


We now intend to prove the following.

\begin{theorem}\label{theo0.2.3new}
Let $R$ be a local ring. Then, the following are equivalent:
\begin{enumerate}
\item
$R$ is a UWNC ring.
\item
$R$ is a weakly nil-clean ring.
\end{enumerate}
\end{theorem}

\begin{proof}
(i) $\Rightarrow$ (ii). As $R$ is local, one finds that ${\rm Id}(R)=\{ 0,1\}$, and so $R$ is abelian. Therefore, $R$ is WUU owing to Corollary \ref{cor0.2.5}. Also, $R$ is clean, because every local ring is clean. Thus, $R$ is a clean WUU and we apply \cite[Corollary 2.15]{4} to find that $R$ is strongly weakly nil-clean and so it is weakly nil-clean.\\
(ii) $\Rightarrow$ (i). It is clear.
\end{proof}

We say that $B$ is a unital subring of a ring $A$ if $\emptyset \neq B\subseteq A$ and, for any $x,y\in B$, the relations $x-y$, $xy\in B$ and $1_{A}\in B$ hold. Let $A$ be a ring and $B$ a unital subring of $A$, and denote by $R[A,B]$ the set $\lbrace (a_{1},\ldots ,a_{n},b,b,\ldots ): a_{i}\in A,b\in B,1\leq i\leq n \rbrace$. Then, $R[A,B]$ forms a ring under the usual component-wise addition and multiplication.

\medskip

We, thereby, establish the following.

\begin{proposition}\label{prop0.2.9}
Let $A$ be a ring and a unital subring $B$ of $A$. Then, the following are equivalent:
\begin{enumerate}
\item
$A$ and $B$ are UWNC.
\item
$R[A,B]$ is UWNC.
\end{enumerate}
\end{proposition}

\section{UWNC Group Rings}

We begin here with the following simple but useful technicality.

\begin{lemma}\label{lem0.3.1}
Let $R$ and $S$ be rings and $i:R\rightarrow S$, $\varepsilon :S\rightarrow R$ be ring homomorphisms such that $\varepsilon i=id_{R}$.
\begin{enumerate}
\item
$\varepsilon ({\rm Nil}(S))={\rm Nil}(R)$, $\varepsilon (U(S))=U(R)$ and $\varepsilon ({\rm Id}(S))={\rm Id}(R)$.
\item
If $S$ is a UWNC ring, then $R$ is a UWNC ring.
\item
If $R$ is a UWNC ring and $\text{ker}\varepsilon \subseteq {\rm Nil}(S)$, then $S$ is a UWNC ring.
\end{enumerate}
\end{lemma}

\begin{proof}
\begin{enumerate}
\item
Clearly, the inclusions $\varepsilon ({\rm Nil}(S))\subseteq {\rm Nil}(R)$, $\varepsilon (U(S))\subseteq U(R)$ and $\varepsilon ({\rm Id}(S))\subseteq {\rm Id}(R)$ are valid. On the other hand, we also have that ${\rm Nil}(R)=\varepsilon i({\rm Nil}(R))\subseteq \varepsilon ({\rm Nil}(S))$, $U(R)=\varepsilon i(U(R)) \subseteq \varepsilon (U(S))$ and ${\rm Id}(R)=\varepsilon i({\rm Id}(R))\subseteq \varepsilon ({\rm Id}(S))$.

\medskip

\item
{\bf Method 1:} Let $S$ be a UWNC ring. Choose $u\in U(R)=\varepsilon (U(S))$, so $u=\varepsilon (v)$, where $v\in U(S)$. Thus, we can write $v=\pm e+q$, where $e\in {\rm Id}(S)$, $q\in {\rm Nil}(S)$. Therefore, $$u=\varepsilon (v)=\varepsilon (\pm e+q)=\pm \varepsilon (e)+\varepsilon (q),$$ where $\varepsilon (e)\in {\rm Id}(R)$ and $\varepsilon (q)\in {\rm Nil}(R)$ exploiting (i).\\

\medskip

{\bf Method 2:} Let $S$ be a UWNC ring. Hence, $U(S)=\pm {\rm Id}(S)+{\rm Nil}(S)$, whence by (i) one has that $$U(R)=\varepsilon (U(S))=\varepsilon (\pm {\rm Id}(S)+{\rm Nil}(S))=\pm \varepsilon ({\rm Id}(S))+ \varepsilon ({\rm Nil}(S))=\pm {\rm Id}(R)+{\rm Nil}(R),$$ as required.
\item
If $R$ is a UWNC ring, point (i) enables us that $$U(S)=\varepsilon^{-1}(U(R))=\varepsilon^{-1}(\pm {\rm Id}(R)+{\rm Nil}(R))=\pm {\rm Id}(S)+{\rm Nil}(S)+\text{ker}\varepsilon =\pm {\rm Id}(S)+{\rm Nil}(S),$$ as required.
\end{enumerate}
\end{proof}

\begin{remark}\label{remark0.3.1}
It is a routine technical exercise to see that Lemma \ref{lem0.3.1} (2) and (3) is true also for WUU rings.
\end{remark}

Suppose now that $G$ is an arbitrary group and $R$ is an arbitrary ring. As usual, $RG$ stands for the group ring of $G$ over $R$. The homomorphism $\varepsilon :RG\rightarrow R$, defined by $\varepsilon (\displaystyle\sum_{g\in G}a_{g}g)=\displaystyle\sum_{g\in G}a_{g}$, is called the {\it augmentation map} of $RG$ and its kernel, denoted by $\Delta (G)$, is called the {\it augmentation ideal} of $RG$.

\begin{proposition}\label{prop0.3.1}
Let $R$ be a ring and $G$ a group. If $RG$ is a UWNC ring, then $R$ is a UWNC ring. The converse holds if $\Delta (G)\subseteq {\rm Nil}(RG)$.
\end{proposition}

\begin{proof}
Let us consider the inclusion $i: R\rightarrow RG$, given by $i(r)=\displaystyle\sum_{g\in G}a_{g}g$, where $a_{1G}=r$ and $a_{g}=0$ provided $g\neq 1_{G}$. It is easy to check that the map $i$ is a ring monomorphism and thus $R$ can also be viewed as a subring of $RG$. Furthermore, it is only enough to apply Lemma \ref{lem0.3.1} (ii) to get the claim.
\end{proof}

\begin{corollary}\label{cor0.3.1}
Let $R$ be a ring and $G$ a group. If $RG$ is a WUU ring, then $R$ is a WUU ring. The converse holds if $\Delta (G)\subseteq {\rm Nil}(RG)$.
\end{corollary}

\begin{proof}
It follows at once by Proposition \ref{prop0.3.1} and Lemma \ref{lem0.3.1}.
\end{proof}

A group $G$ is called {\it locally finite} if every finitely generated subgroup of $G$ is finite. Let $p$ be a prime number. A group $G$ is called a {\it $p$-group} if the order of each element of $G$ is a power of $p$.

\medskip

We finish off our results with the following statement.

\begin{proposition}\label{prop0.3.2}
Let $R$ be a UWNC ring with $p\in {\rm Nil}(R)$ and let $G$ be a locally finite $p$-group, where $p$ is a prime. Then, $RG$ is a UWNC ring.
\end{proposition}

\begin{proof}
Referring to \cite[Proposition 16]{11}, one verifies that $\Delta (G)$ is nil. Now, the assertion follows from the obvious isomorphism $\dfrac{RG}{\Delta (G)}\cong R$ and Theorem \ref{theo2.7}.
\end{proof}

\begin{remark}\label{remark0.3.2}
It is easily seen that Proposition \ref{prop0.3.2} holds also for WUU rings.
\end{remark}

\section{Open Questions}

We close the work with the following challenging conjectures and problems.\\

\medskip

A ring $R$ is called {\it uniquely weakly nil-clean}, provided that $R$ is a weakly nil-clean ring in which every nil-clean element is uniquely nil-clean (see \cite{2}).

\begin{conjecture}
A ring R is a WUU ring if, and only if, every unit of $R$ is uniquely weakly nil-clean.\\
\end{conjecture}

\begin{conjecture}
A ring $R$ is a strongly weakly nil-clean if, and only if, it is a semi-potent WUU ring.\\
\end{conjecture}

\begin{problem}
Is a clean, UWNC ring a weakly nil-clean ring?\\
\end{problem}

\begin{problem}
Characterize semi-perfect UWNC rings. Are they weakly nil-clean?\\
\end{problem}

\begin{problem}
Suppose that $R$ is a ring and $n\in \mathbb{N}$. Find a criterion when the full $n\times n$ matrix ring ${\rm M}_{n}(R)$ is UWNC.
\end{problem}

\medskip
\medskip

\noindent{\bf Funding:} The work of the first-named author, P.V. Danchev, is partially supported by the project Junta de Andaluc\'ia under Grant FQM 264, and by the BIDEB 2221 of T\"UB\'ITAK.


\end{document}